\documentclass{amsart}
\usepackage{graphicx} 
\usepackage{latexsym}
\usepackage{amssymb}
\usepackage{amsmath}
\usepackage{amsthm, url}
\usepackage{amscd}
\theoremstyle{plain}

\usepackage{tikz, tikz-cd}

\usepackage{cite}
\usepackage{color}
\usepackage[pagebackref,colorlinks,citecolor=blue,linkcolor=red]{hyperref}

\renewcommand*{\backrefalt}[4]{\ifcase #1 (Not cited).\or (Cited p.~#2).\else (Cited pp.~#2).\fi} 
\usepackage{float}
\usepackage{tikz}
\usepackage{lmodern}
\usetikzlibrary{decorations.pathmorphing}

\theoremstyle{definition}
\newtheorem {theorem}{Theorem}[section]
\newtheorem{rem}[theorem]{Remark}

\newtheorem{ques}[theorem]{Question}
\newtheorem{prop}[theorem]{Proposition}

\newtheorem{exam}[theorem]{Example}
\newtheorem{lemma}[theorem]{Lemma}
\newtheorem{defin}[theorem]{Definition} 
\newtheorem{cor}[theorem]{Corollary}

\title[]{Linearity and virtual poly-freeness of the fundamental group of plane curves of degree at most five}

\author{Shengkui Ye}
\address{NYU Shanghai, No.567 Yangsi West 
  Rd, Pudong New Area, Shanghai, 200124, P.R. China \\
NYU-ECNU Institute of Mathematical Sciences at NYU Shanghai, 3663 Zhongshan Road North, Shanghai, 200062, China}
\email{sy55@nyu.edu}

\author{Kejia Zhu}
\address{Department of Mathematics, University of California at Riverside}
\email{kzhumath@gmail.com}

\begin{document}
\maketitle

\begin{abstract}
We prove that for any algebraic plane curve $C$ of degree at most $5$, the fundamental group $\pi_1(\mathbb CP^2\setminus C)$ is linear and virtually polyfree. As a consequence, we answer positively the open question on the residual finiteness of these groups for all plane curves of degree at most $5$.
\end{abstract}

\section{Introduction}
Let $C$ be an algebraic curve in the projective plane $\mathbb CP^2$.  We consider the following question.

\begin{ques}
\label{q1}
    Is the fundamental group $\pi_1(\mathbb CP^2\setminus C)$ of an algebraic plane curve residually finite?
\end{ques}

This question has been open for a long time. Zariski had already posed the question of the existence of non-residually finite groups, cf. § 1 and Appendix 1) of Chapter VIII of \cite{Zar} (see also Libgober \cite[Problem 3.1] {Lib}).
Even in the low-degree cases (degree up to five), such a question was still unknown (see the survey article of Artal Bartolo-Cogolludo Agustın-Tokunaga \cite{bartolo2008survey}, page 3). Toledo \cite{Toledo} constructed a smooth projective variety with a non-residually finite fundamental group, answering in the negative a question of Serre. However, it is not clear whether one can choose the example as the complement of an algebraic curve with a non-residually finite fundamental group. Koberda and Suciu \cite[Problem 1.9]{KS} asked Question \ref{q1} for affine line arrangements.

The purpose of this article is to answer Question \ref{q1} positively for algebraic curves of degree at most five. Recall that a group $G$ is linear if it can be embedded into a general linear group $\mathrm{GL}(n,F)$ for some field $F$ and positive integer $n$. A linear group is residually finite. A group $G$ is called poly-free if there is a finite
subnormal sequence%
\begin{equation*}
1<G_{1}\trianglelefteq G_{2}\trianglelefteq \cdots \trianglelefteq G_{n}=G
\end{equation*}%
such that the quotient $G_{i}/G_{i+1}$ is free for each $i$. Polyfree groups have many nice properties, like left-orderable, having Tits alternative, and so on. We say a group $G$ virtually has property P, if there is a finite-index subgroup $H<G$ that has the property P. The following is our main result.

\begin{theorem}
\label{main}
    Let $C\subset\mathbb CP^2$ be an algebraic plane curve of degree at most five. Then the fundamental group $\pi_1(\mathbb CP^2\setminus C)$ is linear and virtually polyfree. In particular, $\pi_1(\mathbb CP^2\setminus C)$ is residually finite.
\end{theorem}

It is well-known (cf.\cite[2.3.2]{cogolludo2011braid}) that for any plane curve $C\subset\mathbb CP^2$ of degree $\le3$, the fundamental group $\pi_1(\mathbb CP^2\setminus C)$ is either finitely generated abelian or a free group $\mathbb Z\ast\mathbb Z$. Thus $\pi_1(\mathbb CP^2\setminus C)$ is linear. For degree-$4$ curves, we applied the results due to Nori \cite{nori1983zariski} and Cogolludo-Agustin and Elduque  \cite{cogolludo2025quasi} to conclude that
that $\pi_1(\mathbb{C}P^2\setminus C)$ must be one of the following groups: a (virtually) finitely generated abelian group, a finite group, the braid group on the sphere $B_3(S^2)$, the free group $F_3$, the braid group $B_3$, the free product $\mathbb{Z} \ast \mathbb{Z}/2\mathbb{Z}$, or the semi-direct product $F_2 \rtimes \mathbb{Z}$, (see Subsection \ref{degree4} for details), which are all virtually poly-free and linear. For degree-$5$ curves, Degtyarev \cite{degtyarev1987topology} obtained presentations of all the $\pi_1(\mathbb{C}P^2\setminus C)$ when $\deg C=5$. In Section \ref{degree5}, we will review the presentations of these groups from Degtyarev and prove the virtual polyfreeness and linearity of these fundamental groups. Our key observation is that most of the non-abelian groups are virtual surface groups or (triangle or right-angled) Artin groups.

\subsection*{Acknowledgments}
We would like to thank Anatoly Libgober, Jose Ignacio Cogolludo, Corey Bregman for useful discussions and Alex Degtyarev for pointing us the results from his thesis and book.

\section{ Basic notations and facts}
We introduce some basic notations and facts that will be used in the subsequent sections.
\subsection{Poly-free groups}

Recall that a group $G$ is called poly-free if there is a finite
subnormal sequence%
\begin{equation*}
1<G_{1}\trianglelefteq G_{2}\trianglelefteq \cdots \trianglelefteq G_{n}=G
\end{equation*}%
such that the quotient $G_{i}/G_{i+1}$ is free for each $i$. Here the rank of $%
G_{i}/G_{i+1}$ is not required to be finite. When each quotient $%
G_{i}/G_{i+1}$ is cyclic (or infinite cyclic), the group $G$ is called polycyclic (resp.
poly-$\mathbb{Z}$). The group $G$ is called normally poly-free if
additionally each $G_{i}$ is normal in $G$.

Poly-free groups have many nice properties. For example, they are torsion-free, locally indicable,  having finite asymptotic dimension, satisfying the Baum--Connes Conjecture with coefficients, satisfying the Farrell--Jones Conjecture if they are normally
poly-free.  Furthermore,  a poly-free group satisfies the Tits alternative, i.e.~any nontrivial
subgroup either is polycyclic or contains a nonabelian free subgroup. For more details, see \cite{YW}. 

\subsection{Artin groups}

 Let $\Gamma =(V,E)$ be a complete graph with a vertex set $V$ and edge set $E$, and each edge $e$ is labelled by an
integer $m_{e}\in \{2,3,....,\infty \}.$
 The Artin group $\mathrm{Art}_{\Gamma }$ is given by the presentation%
\begin{eqnarray*}
\langle a,b &\in &V\mid \{a,b\}^{m_{ab}}=\{b,a\}^{m_{ab}}, [a,b] \in E\rangle ,
\end{eqnarray*}%
where $\{a,b\}^{m}=ababa...$, a word of length $m.$ When $m=\infty$, we assume that there is no such relator. Similarly, the Coxeter group $\mathrm{Cox}_\Gamma$ is defined by the presentation
\begin{eqnarray*}
\langle a,b &\in &V\mid \ a^2=1, \{a,b\}^{m_{ab}}=\{b,a\}^{m_{ab}}, 
[a,b] \in E\rangle.
\end{eqnarray*}
There is an obvious epimorphism $\phi: \mathrm{Art}_\Gamma \rightarrow \mathrm{Cox}_\Gamma$ mapping generators to the corresponding generators. An Artin group $\mathrm{Art}_\Gamma$ is of finite type, if the corresponding Coxeter group $\mathrm{Cox}_\Gamma$ is finite.

When each $m_{e}=\infty ,$ the group $\mathrm{Art}_{\Gamma }$ is free.  When each $m_{e}=2,$ the group $\mathrm{Art}_{\Gamma }=\mathbb{Z}%
^{|V|} $. When each $m_{e}\in \{2,\infty \},$ $\mathrm{Art}_{\Gamma }$ is a
Right-Angled Artin group (RAAG). It is well-known that an RAAG is isomorphic to a subgroup of the special linear group $\mathrm{SL}(n,\mathbb{Z})$. Hermiller–{\v{S}}iuni{\'c} \cite{hermiller2007poly} proved that a right-angled Artin group is always poly-free. When $\Gamma$ is a complete graph on $n$ vertices with a spanning path having each edge labeled by $3$ and all other edges are labeled by $2$, the $\mathrm{Art}_{\Gamma }$ is the braid group $B_{n+1}$ of $n+1$ strings. If $\Gamma$ is a triangle with three vertices and three edges labeled by $M,N,P$, we will use the $\mathrm{Art}_{MNP}$ to denote $\mathrm{Art}_{\Gamma }$. Bigelow \cite{Big} and Krammer \cite{Kra} have proved that all braid groups $B_n$ are linear. It is also known that any finite-type Artin group is linear (see \cite{CW}).  An Artin group $\mathrm{Art}_{B_n}$ of type $B_n$ is isomorphic to the semi-direct product $\mathrm{Art}_{\tilde{A}_{n-1}} \rtimes \mathbb{Z}$, where $\mathrm{Art}_{\tilde{A}_{n-1}}$ is the Artin group of affine type  (cf. Charney-Crisp \cite[page 6]{CC}).  Therefore, the Artin group $\mathrm{Art}_{\tilde{A}_{n-1}}$ of affine type is linear.  Squier \cite{Sq} proved that the affine triangle Artin group $\mathrm{Art}_{333},\mathrm{Art}_{244},\mathrm{Art}_{236} $ are polyfree.

\subsection{Algebraic curves and blow up}
The following are some basic notations and facts from algebraic geometry.

\begin{defin}
A plane curve in $\mathbb CP^2$ is given by the zero locus of a homogeneous polynomial $f(x,y,z)\in \mathbb C[x,y,z]$, denoted by $V(f(x,y,z)$.\\
$\bullet$ A curve $V(f)$ is reducible, if $f=f_1\cdot f_2$, so that $V(f)=V(f_1)\cup V(f_2)$. \\
$\bullet$ A curve $V(f)$ is reduced, if $f$ doesn't have repeated factors.\\
$\bullet$ A curve in $\mathbb CP^2$ is conic, if its degree is $2$; A curve in $\mathbb CP^2$ is cubic, if its degree is $3$; A curve in $\mathbb CP^2$ is quartic, if its degree is $4$; A curve in $\mathbb CP^2$ is quintic, if its degree is $5$.
\end{defin}

\begin{rem}
    For a conic curve, it's irreducible if and only if it's smooth. For a cubic curve, there are exactly $3$ cases: (1) it's smooth; (2) it has a node as the unique singularity; (3) it has a cusp as the unique singularity. Here, a node means it's locally given by the equation $y^2-x^2$ (in an affine chart), a cusp means it's locally given by the equation $y^2-x^3$ (in an affine chart).
\end{rem}

\begin{defin}
    Let $Y$ be a complex surface and $p \in Y$ a smooth point of $Y$. Choose a coordinate neighborhood $(U, z)$ centered at $p$ such that $z(p) = 0$ and $z(U) \subset \mathbb{C}^2$. Define:
\[
U' = \{(x, \ell) \in U \times \mathbb{CP}^1 : x \in \ell\},
\]
where $\ell$ is a line through the origin in $\mathbb{C}^2$ (under the identification $z: U \to \mathbb{C}^2$). Let $\pi: U' \to U$ be the projection $(x, \ell) \mapsto x$. Then:
\begin{itemize}
    \item $\pi$ restricts to a biholomorphism $U' \setminus \pi^{-1}(p) \to U \setminus \{p\}$.
    \item The \emph{blow-up of $Y$ at $p$} is defined as:
    \[
    \mathrm{Bl}_p(Y) = (Y \setminus \{p\}) \cup U',
    \]
    where the union is taken by identifying each $(x,l) \in U' \setminus \pi^{-1}(p)$ with $x \in U \setminus \{p\}$.
    \item The \emph{blow-up projection} $\pi: \mathrm{Bl}_p(Y) \to Y$ is defined as:
   $\pi([y]) = y$ for $y \in Y \setminus \{p\}$; $\pi([(x, \ell)]) = x$ for $(x, \ell) \in U'$.
    \end{itemize}
 The preimage $\pi^{-1}(p)\cong\mathbb CP^1$ is called an exceptional divisor.
\end{defin}

\begin{rem}
Topologically, the $\mathrm{Bl}_p(Y)$ above is diffeomorphic to $Y\#\overline{\mathbb CP^2}$, thus blowing up at a point will not change the fundamental group. When we blow up at a node, then the exceptional divisor and the 2 branches are transverse with 2 intersections; when we blow up at a cusp, the 2 branches merge into one smooth branch and tangent to the exceptional divisor, forming a tacnode, locally given by $y^2-x^4$. See Figure \ref{fig:blowup}. For more details, see \cite{harris2013algebraic}.
\end{rem}

\begin{figure}[h]
    \centering
\begin{tikzpicture}[scale=1.8]
    
    \begin{scope}[blue]
        \draw[thick, domain=-0.8:0.8, smooth, samples=100] 
            plot ({\x*\x}, {\x*\x*\x});
        
        \fill[orange] (0,0) circle (1.5pt) node[below left] {};
        \node[above] at (0.4,0.7) {\textbf{Cusp}};
    \end{scope}

    \begin{scope}[blue, shift={(2,0)}]
        \draw[thick, domain=-0.8:0.8, smooth, samples=100] 
            plot ({{\x*\x},\x});
        
        \fill (0,0) circle (1.5pt) node[below left] {};
        \node[left] at (-0.1,0) {\textbf{Tacnode}};
    \end{scope}
    \begin{scope}[orange, shift={(2,0)}]
        \draw[thick, domain=-0.8:0.8, smooth, samples=100] 
            plot ({0,{\x}});
        
    \end{scope}
\begin{scope}[red,shift={(4,0)}]
        \draw[thick, domain=-1.15:1.15, smooth, samples=100] 
            plot ({\x*\x - 1}, {\x*\x*\x - \x});
        
        \fill[cyan] (0,0) circle (1.5pt) node[below] {};
        \node[right] at (-0.55,0.6) {\textbf{Node}};
    \end{scope}

\begin{scope}[red,shift={(5,0)}]
        \draw[thick, domain=-0.6:0.6, smooth, samples=100] 
            plot ({3*\x*\x},{\x});
        \node[right] at (0.8,0) {\textbf{Transversal}};
    \end{scope}

 \begin{scope}[cyan, shift={(5.7,0)}]
        \draw[thick, domain=-0.8:0.8, smooth, samples=100] 
            plot ({0,{\x}});
        
    \end{scope}
    
\end{tikzpicture}
\caption{The blowup of a cusp and a node.}    \label{fig:blowup}
\end{figure}
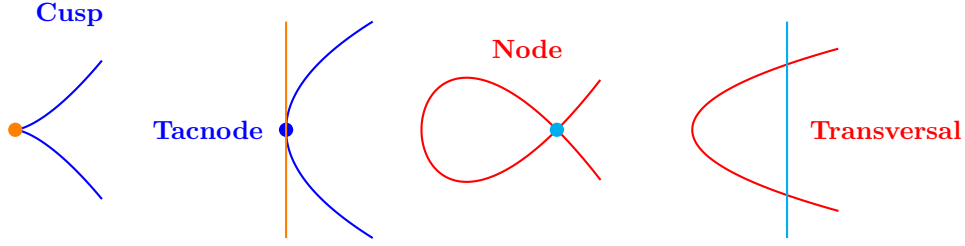

\begin{defin}
    Given a blow up $\pi:X\to Y$ at a point $p$ contained in a curve $C\subset Y$, the strict transform of $C$, is defined to be $\tilde C:=\overline{\pi^{-1}(C\setminus\{p\})}$.  The total transform of $C$ is $\pi^*(C):=\pi^{-1}(C)\subset X$. 
\end{defin}

\begin{rem}\label{Rem:BlowExp}
    If we blow up at a point $P$ of a curve $C$, then $\pi^*(C)=\tilde C+mE$, where $E$ is the exceptional divisor and $m$ is the multiplicity of $C$ at $P$. In particular, if $C$ is smooth at $P$, then $m=1$; if $P$ is a node or a cusp, then $m=2$. See \cite[Lemma II.2]{beauville1996complex} for details.
\end{rem}

\begin{defin}
    Given two curves $C, D$ in a complex surface $X$, the intersection number $C\cdot D$ is defined as follows: choose a nonsingular curve $C'$ linear equivalent to $C$, and a nonsingular curve $D'$ linear equivalent to $D$ and transversal to $C'$. Define $C\cdot D$ to be the number of intersections of $C'$ and $D'$.
\end{defin}
 In particular, $C\cdot C$ is the self-intersection number of $C$ when $D=C$. Given a plane curve $C\subset \mathbb CP^2$ of degree $d$, $C\cdot C=d^2$ (counting multiplicity) (c.f. \cite[V.1.4.2, page 360]{hartshorne2013algebraic}). Furthermore, when we blow up at a singular point of a cubic/conic curve $C$,   we have the following lemma following from \cite[Proposition I.8]{beauville1996complex}. We will use it for self-intersections in Proposition \ref{NormalCrossing}.
 
\begin{lemma}
    The self-intersection number of the strict transform $\tilde C$ will have the relation: if we blow at a cusp or node, $\tilde C\cdot \tilde C=C\cdot C-4$, while if we blow up at a smooth point, $\tilde C\cdot \tilde C=C\cdot C-1$.
\end{lemma}
\begin{proof} By \cite[Proposition I.8]{beauville1996complex}, we have
    \[
\pi^*(C) \cdot \pi^*(C) = C \cdot C.
\]
By \(\pi^*(C) = \widetilde{C} + mE\) (see Remark \ref{Rem:BlowExp}), we have

\[
\widetilde{C}^2 + 2m(\widetilde{C} \cdot E) + m^2E^2.
\]
Since \(\widetilde{C} \cdot E = m\) and \(E^2 = -1\), we have
\[
\widetilde{C}^2 + 2m^2 +m^2(-1) = C\cdot C,
\]
\[
\widetilde{C}\cdot \widetilde{C}=C\cdot C - m^2.
\] 
By Remark \ref{Rem:BlowExp}, the result follows.
\end{proof}

 \begin{defin}\cite[see Remark 3 for more details]{bartolo2008survey}
The combinatorial type of a curve $C$ is given by a 7-tuple:\[
\Bigl( \mathrm{Irr}(C),\: \deg,\: \mathrm{Sing}(C),\: \Sigma_{\mathrm{top}}(C),\: \sigma_{\mathrm{top}},\: \{ C(P) \}_{P \in \mathrm{Sing}(C)},\: \{ \beta_{P} \}_{P \in \mathrm{Sing}(C)} \Bigr),
\]
where:
\begin{itemize}
    \item $\mathrm{Irr}(C)$ is the set of irreducible components of $C$ and $\deg \colon \mathrm{Irr}(C) \to \mathbb{Z}$ assigns to each irreducible component its degree.
    
    \item $\mathrm{Sing}(C)$ is the set of singular points of $C$, $\Sigma_{\mathrm{top}}(C)$ is the set of topological types of $\mathrm{Sing}(C)$, where the topological type of a singularity is the homeomorphism type of its link, and $\sigma_{\mathrm{top}} \colon \mathrm{Sing}(C) \to \Sigma_{\mathrm{top}}(C)$ assigns to each singular point its topological type.
    
    \item $C(P)$ is the set of local branches of $C$ at $P \in \mathrm{Sing}(C)$ and $\beta_{P} \colon C(P) \to \mathrm{Irr}(C)$ assigns to each local branch the global irreducible component containing it.
\end{itemize}
\end{defin}

\begin{rem}\label{Zariskipair}
It is well known that, for the curves of degree $\le5$, if two curves $C_a, C_b$ have the same combinatorial type, then $\pi_1(\mathbb CP^2\setminus C_a)\cong \pi_1(\mathbb CP^2\setminus C_b)$ (it is recorded in Degtyarev's thesis \cite{degtyarev1987topology}). As a consequence, for a specific combinatorial type, it suffices to consider the calculation via a specific equation whose zero locus satisfies the given combinatorial type.
\end{rem}

\subsection{Presentation of subgroups: Reidemeister-Schreier method}

Let $G=\langle a_{i},i\in I|r_{i},i\in J\rangle $ be a presentation of a
group, where $I,J$ are two index sets.  For a subgroup $H\leq G$,  a system $R$
of words in the generators $a_{i},i\in I$, is called a Schreier system for $G
$ modulo $H$ if 

\begin{itemize}
\item[(i)] every right coset of $H$ in $G$ contains exactly one word of $R$
(i.e.~$R$ forms a system of right coset representatives); 

\item[(ii)] for each word in $R$ any initial segment is also in $R$ (i.e.~initial segments of right coset representatives in $R$ are again right coset
representatives). 

\item[(iii)] the empty word $\emptyset \in R.$
\end{itemize}

For any $g\in G,$ let $\bar{g}$ be the unique element in $R$ such that $Hg=H%
\bar{g}.$ For each $K\in R,a=a_{i},$ let $s_{K,a}=Ka\overline{Ka}^{-1}\in H.$
For more details on the Schreier system, see \cite[Section 2.3]{mks}. The
following Reidemeister theorem gives a presentation of the subgroup $H$.

\begin{theorem}
\label{schreier}
\cite[Corollary 2.7.2 + Theorem 2.8, Theorem 2.9]{mks} The subgroup $H$ has a
presentation%
\begin{eqnarray*}
\langle s_{K,a},K &\in &R,a\in \{a_{i},i\in I\}\mid s_{K,a}=1,K\in R,a\in
\{a_{i},i\in I\},\text{if }Ka\equiv \overline{Ka}, \\
\tau (Kr_{i}K^{-1})& =&1,i\in J,K\in R\rangle ,
\end{eqnarray*}%
where $Ka\equiv \overline{Ka}$ means that the two words are equivalent in the
free group $F(\{a_{i},i\in I\})$, and $\tau \ $is the
Reidemeister--Schreier rewriting function defined as follows
\begin{eqnarray*}
\tau  &:&F(\{a_{i},i\in I\})\rightarrow F(\{s_{K,a},K\in R,a\in \{a_{i},i\in
I\}\}), \\
a_{i_{1}}^{\varepsilon _{1}}a_{i_{2}}^{\varepsilon _{2}}\cdots
a_{i_{m}}^{\varepsilon _{m}} &\mapsto &s_{K_{i_{1}},a_{i_{1}}}^{\varepsilon
_{1}}s_{K_{i_{2}},a_{i_{2}}}^{\varepsilon _{2}}\cdots
s_{K_{i_{m}},a_{i_{m}}}^{\varepsilon _{m}}
\end{eqnarray*}
with%
\begin{equation*}
K_{i_{j}}=\left\{ 
\begin{array}{c}
\overline{a_{i_{1}}^{\varepsilon _{1}}a_{i_{2}}^{\varepsilon _{2}}\cdots
a_{i_{j-1}}^{\varepsilon _{j-1}}},\text{ if }\varepsilon
_{j}=1, \\ 
\overline{a_{i_{1}}^{\varepsilon _{1}}a_{i_{2}}^{\varepsilon _{2}}\cdots
a_{i_{j}}^{\varepsilon _{j}}},\text{ if }\varepsilon _{j}=-1.%
\end{array}%
\right. 
\end{equation*}
\end{theorem}

\section{Curves of degree $\le4$}
In this section, we will show that given any plane curve $C$ of degree $\le4$, $\pi_1(\mathbb CP^2\setminus C)$ is linear, thus residually finite. We only consider the case of reduced curves, since if there are double curves $V(f^2(x,y,z))$, topologically it reduces to the lower degree case.

\subsection{Curves of degree $\le3$}
\noindent\\
It is well-known (cf.\cite[2.3.2]{cogolludo2011braid}) that for any plane curve $C\subset\mathbb CP^2$ of degree $\le3$, $\pi_1(\mathbb CP^2\setminus C)$ is either finitely generated abelian or a free group $\mathbb Z\ast\mathbb Z$. Thus $\pi_1(\mathbb CP^2\setminus C)$ is linear and virtually polyfree.

\subsection{Quartic curves}\label{degree4}
\noindent\\
If $C\subset \mathbb CP^2$ is an irreducible quartic curve that is not a 3-cuspidal quartic, then the fundamental group $\pi_1(\mathbb CP^2\setminus C)$ is abelian (cf. \cite[p.130, 4.3]{dimca2012singularities}). When $C$ is a 3-cuspidal quartic $\pi_1(\mathbb CP^2\setminus C) \cong B_3(S^2)$, the sphere braid group (cf. \cite{zariski1936poincare}), a finite group of order 12  and thus linear (see also\cite[4.8, page 131]{dimca2012singularities}). Since any finitely generated abelian group is linear, we know that for any irreducible quartic curve $C$,  $\pi_1(\mathbb CP^2\setminus C)$ is linear. Also, these groups are known to be virtually polyfree.

For the reducible quartic curve $C$, although people have the techniques to calculate the fundamental groups, unfortunately, the comprehensive results have not been recorded anywhere. In this subsection, we will list all the cases with proofs.


We will apply the result from Nori to show that some $\pi_1(\mathbb CP^2\setminus C)$'s are abelian.


\begin{prop}\cite[Proposition 3.27]{nori1983zariski}\label{NormalCrossing}
   Let $D$ and $E$ be curves on a non-singular surface $X$. Assume that $D$ has nodes as the only singularities, that $D$ and $E$ intersect transversally
and that for every irreducible component $C$ of $D$ one has $C\cdot C > 2r(C)$ where $C\cdot C$ is the self-intersection of $C$ and $r(C)$ is the number of nodes on $C$. Then $N = \text{Ker }(\pi_1(X \setminus (D \cup E)) \to \pi_1(X\setminus E))$ is finitely generated abelian. Moreover, the centralizer of $N$ in  $\pi_1(X \setminus (D \cup E))$ is of finite index.
 \end{prop} 

\begin{cor}
\label{cornew}
    Under the condition of Proposition \ref{NormalCrossing}, the fundamental group $\pi_1(X \setminus (D \cup E))$ is virtually abelian if $\pi_1(X\setminus E))$ is cyclic.
\end{cor}

\begin{proof}
    Suppose that $\pi_1(X\setminus E))$ is generated by an element $a$. Choose a preimage $t\in \pi_1(X \setminus (D \cup E))$ of $a$. Proposition \ref{NormalCrossing} implies that the fundamental group $\pi_1(X \setminus (D \cup E))$ is isomorphic to the semi-direct product $N \ltimes \langle t \rangle.$ Since the centralizer of $N$ is of finite index, there is an integer $n$ such that $t^n$ lies in the centralizer. This implies that $N \times \langle t^n \rangle$ is an abelian subgroup of finite index. 
\end{proof}

In the following, we will give a rough idea of our strategy to apply Proposition \ref{NormalCrossing}:

When we have an irreducible curve $C$ and a curve $L$, which is not necessarily irreducible. Suppose they don't transverse with each other. Our strategy will be taking a resolution consisting of a series of blow-ups $\pi:X\to \mathbb CP^2$, at non-transversal intersections and at the singularities of the union of the curves (it will generate a series of exceptional divisors $E_1,..., E_n$), so that the strict transform under $\pi$, $D:=\tilde C$ is smooth and transverses the union of the strict transform of $L$, say $E_0$ and exceptional divisors, $E:=\cup_{i=0}^n E_i$. Observe that $X\setminus E$ is equal to removing $L$ and a finite union of (singular) points from $C$. Moreover, since we will only blow up at the $C,L$ and their strict transforms, $X\setminus (D\cup E)=\mathbb CP^2\setminus (C\cup L)$.
 
 In particular, if $L$ is a line, then we have $X\setminus E=\mathbb C^2\setminus \{p_1,...,p_m\}$, which is simply connected.  Now by Proposition \ref{NormalCrossing}, $\pi_1(\mathbb CP^2\setminus (C\cup L))=\pi_1(X \setminus (D \cup E))$ is abelian.

Now we give a concrete example for this process:
\begin{exam}\label{example1}
When $C$ is a cuspidal cubic curve (with a cusp at $q\in C$), $L$ is a line tangent (with multiplicity $2$) to $C$ at a smooth point $p$ of $C$. We need to blow up three times to get all curves transverse, see Figure \ref{fig:example1}.

For the first step, we need to first blow up at $p$ and $q$. In this step, we get two exceptional divisors $E_1,E_2$. The strict transforms $\tilde L,\tilde C$ and the exceptional divisor $E_1$ form a triple point $\tilde p$ at their intersection. $E_2$ is tangent to $\tilde C$ at $q$ (the tangent point $q$ is a singularity of the union of the curves, i.e., a tacnode). 

For the second step, we need to blow up at $\tilde p,\tilde q$. Thus, we get two exceptional divisors $E_3 ,E_4$ and a triple point $\tilde q'$, which is the only singularity. 

For the third step, we blow up at $\tilde q'$, to get one more exceptional divisor $E_5$. Now all the curves are transverse to each other.

Now $X$ is the total transform of $\mathbb CP^2$ after the three steps and $D:=\tilde C, E:=\tilde L\cup_{i=1}^5E_i$. Observe $D,E$ intersects transversely, 
$D$ has only one irreducible component, and is smooth. $D\cdot D=C\cdot C-4-1-1-1-1=1>2r(D)=0$ (recall blow up at a singularity (resp. smooth) point of a cubic/conic curve, the self-intersection number would decrease $4$ (resp. $1$). So we can apply Proposition \ref{NormalCrossing}. There are two blowup points from $C$, i.e., $ p$ and $ q$. Note $p\in L$, So $X\setminus E=\mathbb CP^2\setminus (L\cup\{q\})\cong \mathbb C^2\setminus \{q\}$, which is simply-connected. Thus Proposition \ref{NormalCrossing} tells us that $\pi_1(X\setminus (D\cup E))$ is abelian. Observe $X\setminus (D\cup E)=\mathbb CP^2\setminus (C\cup L)$, so $\pi_1(\mathbb CP^2\setminus (C\cup L))$ is abelian.
  
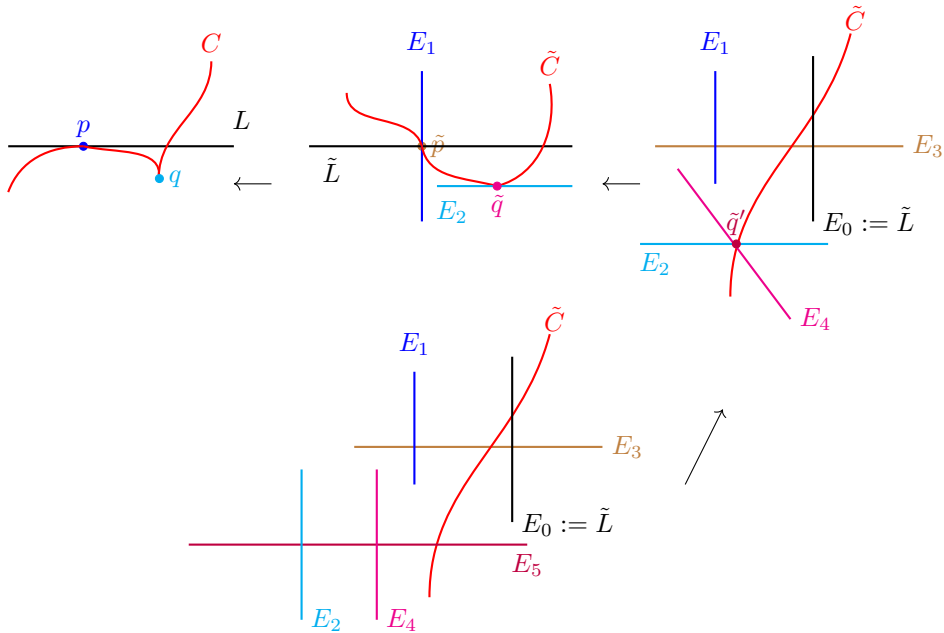
\begin{figure}[h]
    \centering
	\begin{tikzpicture}
		
		\draw[black, thick] (3,0)   to (6,0);
		\filldraw[blue] (4,0) circle (1.5pt) node[above] {$p$};
		\filldraw[black] (6.1,0.1) circle (0pt) node[above] {$L$};
		\filldraw[red] (5.7,1.13) circle (0pt) node[above] {$C$};

		\draw [red,thick](3,-0.61) to[out=70,in=180] (4,0);
        \filldraw[brown] (8.5,0) circle (1.5pt) node[right] {$\tilde p$};
        \draw [red,thick](5,-0.43) to[out=90,in=-90] (5.7,1.13);
        \draw [red,thick](4,0) to[out=-13,in=80] (5,-0.43);
		\draw [<-] (6,-0.5) to (6.5,-0.5);
        \filldraw[cyan] (5.01,-0.43) circle (1.5pt) node[right] {$q$};

		
        
		\draw[black, thick] (7,0)   to (10.5,0);
		\draw [blue, thick] (8.5,-1) to (8.5,1);
        \draw [red,thick](7.5,0.71) to[out=-90,in=105] (8.5,0);
        \draw [red,thick](8.5,0) to[out=-80,in=165] (9.5,-0.53);
        \draw [red,thick](9.5,-0.53) to[out=15,in=-80] (10.2,0.83);
        \draw [cyan, thick] (8.7,-0.53) to (10.5,-0.53);
        \filldraw[magenta] (9.5,-0.53) circle (1.5pt) node[below] {$\tilde q$};
        \filldraw[cyan] (8.9,-0.6) circle (0pt) node[below] {$E_2$};
		
		\filldraw[black] (7.3,-0.6) circle (0pt) node[above] {$\tilde L$};
		\filldraw[red] (10.2,0.83) circle (0pt) node[above] {$\tilde C$};
		\filldraw[blue] (8.5,1.1) circle (0pt) node[above] {$E_1$};
        \draw [<-] (10.9,-0.5) to (11.4,-0.5);

        \draw[brown, thick] (11.6,0)   to (14.9,0);
        \filldraw[brown] (14.9,0) circle (0pt) node[right] {$E_3$};
	\draw [blue, thick] (12.4,-0.5) to (12.4,1);
    \filldraw[blue] (12.4,1.1) circle (0pt) node[above] {$E_1$};
        \draw [red,thick](12.6,-2) to[out=90,in=-105] (14.2,1.5);
        \filldraw[red] (14,1.7) circle (0pt) node[right] {$\tilde C$};
        \draw [black, thick] (13.7,-1) to (13.7,1.2);
        \filldraw[black] (13.7,-1) circle (0pt) node[right] {$E_0:=\tilde L$};
        \draw [cyan, thick] (11.4,-1.3) to (13.9,-1.3);
        \filldraw[cyan] (11.6,-1.3) circle (0pt) node[below] {$E_2$};
        \draw [magenta, thick] (11.9,-0.3) to (13.4,-2.3);
        \filldraw[magenta] (13.4,-2.3) circle (0pt) node[right] {$E_4$};
        \filldraw[purple] (12.68,-1.3) circle (1.5pt) node[above] {$\tilde q'$};

        \draw [<-] (12.5,-3.5) to (12,-4.5);
        \draw[brown, thick] (7.6,-2-2)   to (10.9,-2-2);
        \filldraw[brown] (10.9,-2-2) circle (0pt) node[right] {$E_3$};
	\draw [blue, thick] (8.4,-2.5-2) to (8.4,-1-2);
    \filldraw[blue] (8.4,1.1-2-2) circle (0pt) node[above] {$E_1$};
        \draw [red,thick](8.6,-4-2) to[out=90,in=-105] (10.2,1.5-2-2);
        \filldraw[red] (10,1.7-2-2) circle (0pt) node[right] {$\tilde C$};
        \draw [black, thick] (9.7,-1-2-2) to (9.7,1.2-2-2);
        \filldraw[black] (9.7,-3-2) circle (0pt) node[right] {$E_0:=\tilde L$};
        \draw [purple, thick] (5.4,-3.3-2) to (9.9,-3.3-2);
        \filldraw[purple] (9.9,-3.3-2) circle (0pt) node[below] {$E_5$};
        \draw [magenta, thick] (7.9,-2.3-2) to (7.9,-4.3-2);
        \filldraw[magenta] (7.9,-4.3-2) circle (0pt) node[right] {$E_4$};
        \draw [cyan, thick] (6.9,-2.3-2) to (6.9,-4.3-2);
        \filldraw[cyan] (6.9,-4.3-2) circle (0pt) node[right] {$E_2$};          
	
	\end{tikzpicture}
\caption{The blowup process of Example \ref{example1}.}    \label{fig:example1}
\end{figure}
\end{exam}   

Moreover, if we know a fundamental group is abelian, we can easily get the fundamental group by the following well-known fact (cf.\cite[Corollary 2.8]{cogolludo2011braid}):

\begin{prop}\label{AbelianFund}
    Let $C=C_1\cup C_2\cup...\cup C_r$ be the decomposition of $C$ into its irreducible components. The abelianization $\pi_1(\mathbb CP^2\setminus C)_{ab}=\mathbb Z^{r-1}\oplus \mathbb Z/\tau\mathbb Z$, where $\tau$ is the greatest common divisor of $\text{deg}(C_1),...,\text{deg}(C_r)$.
\end{prop}


In the following, we list the fundamental groups of all combinatorial types with proofs. We will frequently use the basic fact that a degree-$d$ algebraic curve $C$ has at most $\frac{(d-1)(d-2)}{2}$ singular points. This follows from the Noether's degree-genus formula $g=\frac{(d-1)(d-2)}{2} - \sum \delta_p \geq 0$ (cf. \cite{Kirwan}[Theorem 7.37, page 217] ). In particular, a degree-2 (irreducible) algebraic curve has no singularity. A degree-3 algebraic curve has at most one singularity, while a degree-4 curve has at most 3 singular points.\\


\noindent (1) $C=4C_1$ (a union of $4$ lines).

(1.1) If $4$ lines are in general position (i.e., all lines transverse, no triple point), then it's a nodal curve. It's well-known that the fundamental group of a nodal plane curve is abelian (cf. \cite{fulton1980fundamental}). Now by Proposition \ref{AbelianFund}, $\pi_1(\mathbb CP^2\setminus \cup_{i=1}^4L_i)=\mathbb Z^3$.

(1.2) If there is one triple point (i.e. $3$ lines are concurrent), $\pi_1(\mathbb CP^2\setminus \cup_{i=1}^4L_i)=F_2\times \mathbb Z$ (if we project at the triple point, then we get a trivial fibration over $\mathbb CP^1$ with $3$ punctures ($3$ punctures correspond to the $3$ concurrent lines) and the fibers are $\mathbb CP^1$ with $2$ punctures($2$ punctures correspond to the intersection with the concurrent intersection and the intersection with the rest line).

(1.3) If $4$ lines are concurrent, $\pi_1(\mathbb CP^2\setminus \cup_{i=1}^4L_i)=F_3$. Indeed, projecting from a point outside the lines, we see that there is only one singular fiber. So there is no monodromy relation, as there is only one singular fiber. Then by Zariski-Van Kampen theorem (cf.\cite[Theorem 2.6]{cogolludo2011braid}), $\pi_1(\mathbb CP^2\setminus \cup_{i=1}^4L_i)=\langle g_1,g_2,g_3,g_4|g_4g_3g_2g_1=1\rangle=F_3$.\\

\noindent(2) $C=C_3\cup C_1$ (a union of a cubic curve and a line).
Singularity type of $C_3$ is either $A_1$ (node) or $A_2$ (cusp). We consider the position of $C_3$ and $C_1$. The Bezout theorem implies that the intersection number of $C_1$ and $C_2$ is three, counting multiplicity.  

(2.1) $C_3$ is smooth.

(2.1.1) If $C_3$ and $C_1$ are transverse, the union $C_3 \cup C_1$ a nodal curve, implying that $\pi_1(\mathbb CP^2\setminus (C_3\cup C_1))$ is abelian (cf. \cite[Theorem 2.15]{cogolludo2011braid}). It follows from Proposition \ref{AbelianFund}, $\pi_1(\mathbb CP^2\setminus C_3\cup C_1)=\mathbb Z$.

(2.1.2) If $C_3$ and $C_1$ have a tangency at multiplicity $3$ (inflectional tangency), then by a similar analysis as the Example \ref{example1}, we blow up at the tangent point of $C_3$ and $C_1$ $3$ times to make the total transform a nodal curve. Observe that $\tilde C_3\cdot\tilde C_3=9-1-1-1=6>0$. By Proposition \ref{NormalCrossing},
$\pi_1(\mathbb CP^2\setminus (C_3\cup C_1))$ is abelian, and Proposition \ref{AbelianFund} implies
$\pi_1(\mathbb CP^2\setminus (C_3\cup C_1))=\mathbb Z$.


(2.1.3) If $C_3$ and $C_1$ have two intersection points, then by a similar analysis as the Example \ref{example1}, we blow up at the tangent point of $C_3$ and $C_1$ twice to make the total transform a nodal curve. Observe that $\tilde C_3\cdot\tilde C_3=9-1-1=7>0$. By Proposition \ref{NormalCrossing},
$\pi_1(\mathbb CP^2\setminus (C_3\cup C_1))$ is abelian, and Proposition \ref{AbelianFund} implies
$\pi_1(\mathbb CP^2\setminus (C_3\cup C_1))=\mathbb Z$.




(2.2) $C_3$ is nodal.

(2.2.1) If $C_3$ and $C_1$ intersect transversely, by a similar analysis using Proposition \ref{AbelianFund} as the smooth case, we have $\pi_1(\mathbb CP^2\setminus C_3))=\mathbb Z$.

(2.2.2) If $C_3$ and $C_1$ intersect tangentially at a smooth point $p$ of $C_3$ and transversely at another point,  we blow up at $p$ twice. After we blow up $C:=C_3\cup C_1$ at the tangent point $p$ twice, we get that the strict transform of $C_3$ (saying $\tilde C_3$), the strict transform of $C_1$ (saying $E_0$) and the 2 exceptional divisors $\{E_i\}_{i=1}^2$, intersect transversely to each other (normal crossing). Moreover, if we define $X$ to be the resulting manifold obtained from $\mathbb CP^2$ by the 2 times blow ups, then $X\setminus (\cup_{i=0}^2 E_i\cup \tilde C_3)=\mathbb CP^2\setminus (C_3\cup C_1)$ and $X\setminus (\cup_{i=0}^2 E_i)=\mathbb CP^2\setminus C_1$, which is homeomorphic to $\mathbb C^2$, thus simpy-connected. Now we define $D$ to be $\tilde C_3$, which has only one irreducible component and has a unique singularity, i.e., the node, and define $E$ to be $\cup_{i=0}^2 E_i$. Observe that $\tilde C_3\cdot \tilde C_3=7>2r(\tilde C_3)=2$. It follows from Proposition \ref{NormalCrossing}, $\pi_1(\mathbb CP^2\setminus( C_3\cup C_1))$ is abelian. By Proposition \ref{AbelianFund}, $\pi_1(\mathbb CP^2\setminus( C_3\cup C_1))=\mathbb Z$.

(2.2.3) If $C_3$ and $C_1$ have a tangency (at a smooth point of $C_3$) with multiplicity $3$ (inflectional tangency), we can apply the same method as the last case to blow up $C:=C_3\cup C_1$ at the tangent point $p$ (which is smooth at $C_1$ and $C_3$) three times to make the total transform to have normal crossings. Using the same notation, we have $X\setminus (\cup_{i=0}^3 E_i)=\mathbb CP^2\setminus C_1$, which is homeomorphic to $\mathbb C^2$, thus simpy-connected. Observe that $\tilde C_3\cdot \tilde C_3=6>2r(\tilde C_3)=2$. It follows from Proposition \ref{NormalCrossing}, $\pi_1(\mathbb CP^2\setminus( C_3\cup C_1))$ is abelian. By Proposition \ref{AbelianFund}, $\pi_1(\mathbb CP^2\setminus( C_3\cup C_1))=\mathbb Z$.

(2.2.4) If $C_3$ and $C_1$ intersect at the node with multiplicity $2$,  we can blow up at the node once so that all the curves are transverse to each other. The strict transform $\tilde C_3$ of $C_3$  is irreducible and smooth and has self-intersection $3^2-4=5>0$. We apply the method from Example \ref{example1}, $D:=\tilde C_3$, $E$ is defined to be the union of the strict transform of $C_1$ and the exceptional divisor. By Proposition \ref{NormalCrossing} and Proposition \ref{AbelianFund}, it follows that $\pi_1(\mathbb CP^2\setminus( C_3\cup C_1))=\mathbb Z$.

(2.2.5) If $C_3$ and $C_1$ intersect at the node with multiplicity $3$, it means $C_1$ is tangent to one branch of $C_3$ at the node. In this case, we need to blow up at the node twice to get the normal crossing condition (unlike the last case, after the first blow up, we have a triple point, so we need to blow up again). Then the strict transform of $C_3$, $\tilde C_3$ is irreducible and smooth and has self-intersection $3^2-4-1=4>0$. We apply the method from Example \ref{example1}, with $D:=\tilde C_3$, $E$ is defined to be the union of the strict transform of $C_1$ and exceptional divisors. By Proposition \ref{NormalCrossing} and Proposition \ref{AbelianFund}, it follows that $\pi_1(\mathbb CP^2\setminus( C_3\cup C_1))=\mathbb Z$.



(2.3) $C_3$ is cuspidal.

(2.3.1) If $C_3$ and $C_1$ are transversal, then by a similar analysis as the Example \ref{example1}, we blow up at the cusp twice to make the total transform a nodal curve. Observe that $\tilde C_3\cdot\tilde C_3=9-4-1=4>0$. By Proposition \ref{NormalCrossing},
$\pi_1(\mathbb CP^2\setminus (C_3\cup C_1))$ is abelian, and Proposition \ref{AbelianFund} implies
$\pi_1(\mathbb CP^2\setminus (C_3\cup C_1))=\mathbb Z$.

(2.3.2) If $C_3$ and $C_1$ have one tangency of multiplicity $2$ at a smooth point of $C_3$ and transverse at a different point, we blow up $C:=C_3\cup C_1$ at the tangent point $p$ twice, and at the cusp three times, then we get the strict transform of $C_3$, say $\tilde C_3$, the strict transform of $C_1$, say $E_0$ and five exceptional divisors $\{E_i\}_{i=1}^5$, which are transversal to each other. Moreover, if we define $X$ to be the $\mathbb CP^2$ after the five times blow ups, then $X\setminus (\cup_{i=0}^5 E_i\cup \tilde C_3)=\mathbb CP^2\setminus (C_3\cup C_1)$ and $X\setminus (\cup_{i=0}^5\tilde E_i)$ is homeomorphic to $\mathbb C^2$ minus $3$ points, thus simpy-connected. Now we define $D$ to be $\tilde C_3$, which only has one irreducible component and is smooth, and define $E$ to be $\cup_{i=0}^5 E_i$. Observe that $\tilde C_3\cdot \tilde C_3=1>2r(\tilde C_3)=0$ then it follows from Proposition \ref{NormalCrossing}, $\pi_1(\mathbb CP^2\setminus( C_3\cup C_1))$ is abelian. So by Proposition \ref{AbelianFund}, $\pi_1(\mathbb CP^2\setminus( C_3\cup C_1))=\mathbb Z$.

(2.3.3) If $C_3$ and $C_1$ have a tangency at multiplicity $3$ (inflectional tangency) at a smooth point of $C_3$, then we can consider $C_1$ as the line at infinity (e.g. see $C_1=V(z), C_3= V(x^2z-y^3)$), $\mathbb CP^2\setminus( C_3\cup C_1)$ is homeomorphic to $\mathbb C^2\setminus C_3'$, where $C_3'$ is an affine cuspidal cubic, thus by Zariski-Van Kampen theorem $\pi_1(\mathbb CP^2\setminus( C_3\cup C_1))=\pi_1(\mathbb C^2\setminus C_3')=B_3$.

(2.3.4) $C_1$ is a line tangent (with multiplicity $2$) to $C_3$ at a smooth point $p\in C_3$. By Example \ref{example1}, $\pi_1(\mathbb CP^2\setminus( C_3\cup C_1))$ is abelian, and it follows from Proposition \ref{AbelianFund}, that $\pi_1(\mathbb CP^2\setminus( C_3\cup C_1))=\mathbb Z$.


(2.3.5) If $C_3$ and $C_1$ intersect at the cusp with multiplicity $3$, it means $C_1$ is tangent to the 2 branches of $C_3$ at the cusp. In this case, we need to blow up at the cusp 3 times so that the divisors from the total transform are with normal crossings (unlike the nodal case, after the first blow-up, the strict transform $\tilde C_3$ and the exceptional divisor $E$ are tangent, but $\tilde C_1$ and $\tilde C_3$ are transverse). Then the strict transform of $C_3$, $\tilde C_3$ is irreducible and smooth and has self-intersection $3^2-4-1-1= 3>0$. We apply the method from Example \ref{example1}, $D:=\tilde C_3$, $E$ is defined to be the union of the strict transform of $C_1$ and exceptional divisors. By Proposition \ref{NormalCrossing} and Proposition \ref{AbelianFund}, it follows that $\pi_1(\mathbb CP^2\setminus( C_3\cup C_1))=\mathbb Z$.\\


\noindent(3) $C=2C_2$ (a union of two smooth conic curves). Note that a conic curve is irreducible if and only if it's smooth.

(3.1) If the two conics have an intersection of multiplicity $4$, then by \cite[Example 2.20]{cogolludo2011braid}, $\pi_1(\mathbb CP^2\setminus C)=\mathbb Z\ast\mathbb Z/2\mathbb Z$.

(3.2)  If the two conics $C_2^1, C_2^2$ have a tangent point with intersection number 3, we blow up at the intersection point with intersection number 3 three times, to make the normal crossings. By a similar analysis as the case Example \ref{example1} (since we blow up at smooth points three times, the self-intersection number of the strict transform $D$ of $C_2^1$ becomes $4-3=1>0$ and we take $E$ to be the union of the strict transform of $C_2^1$ and the 3 exceptional divisors. We also denote the simply-connected surface after blow-ups of $\mathbb CP^2$ to be $X$), we have the $\text{Ker }((\pi_1(X \setminus (D\cup E))=\pi_1(\mathbb CP^2 \setminus (C_2^1\cup C_2^2)) \to \pi_1(X \setminus E)=\pi_1(\mathbb CP^2\setminus C_2^2))=\mathbb Z/2\mathbb Z)$ is abelian. It follows from Corollary \ref{cornew} that $\pi_1(\mathbb CP^2 \setminus (C_2^1\cup C_2^2))$ is virtually abelian.

(3.3) If the two conics have a unique intersection of multiplicity $2$, with a similar analysis as the last case, $\pi_1(\mathbb CP^2\setminus C)$ is virtually abelian. 

(3.4) If the two conics have two intersections of multiplicity $2$ (i.e., two $A_3$-type singularities in Arnold's notation), we apply \cite[Theorem 5.2]{cogolludo2025quasi}. The theorem implies that if there is a pencil given by $af+bg^2$, where $[a:b]\in \mathbb CP^1$ and $f$ is a smooth conic and $g$ is a line, and if we have two smooth conics $V_1:=V(f), V_2:=V(a_2f+b_2g^2)$, then $\pi_1(CP^2\setminus (V_1\cup V_2)=\mathbb Z\ast(\mathbb Z/2\mathbb Z)$. Recall that for curves of degree $<6$, the fundamental groups are fully determined by the combinatorial type.  Therefore, without losing generality, we may consider a specific case that $f=x^2 + y^2 - z^2$, $g=y$, and we consider two smooth conics: $x^2 + y^2 - z^2$ (corresponding to $[a:b]=[1:0]$) and $x^2 + 2y^2 - z^2$ (corresponding to $[a:b]=[1:1]$). Observe that they are smooth, tangent to each other exactly at two points $[1:0:1]$ and $[-1:0:1]$, and they are both from the pencil $af+bg^2$. By Theorem 5.1 of \cite{cogolludo2025quasi}, $\pi_1(\mathbb CP^2\setminus (V_1\cup V_2)=\mathbb Z\ast(\mathbb Z/2\mathbb Z)$. Moreover, it implies that if we have two smooth conics, $V_1$ and $V_2$, with 2 tangent points, then $\pi_1(\mathbb CP^2\setminus (V_1\cup V_2))=\mathbb Z\ast(\mathbb Z/2\mathbb Z)$.


(3.5) If the two conics $ C_2^1, C_2^2$ have $4$ different intersections, the fundamental group $\pi_1(\mathbb CP^2\setminus C)=\mathbb Z\oplus \mathbb Z/2\mathbb Z$ (indeed, we can deduce that it is abelian for the same reason as case (1.1)). By Proposition \ref{AbelianFund}, it follows that
$\pi_1(\mathbb CP^2\setminus (C_2^1\cup C_2^2))=\mathbb Z\oplus \mathbb Z/2\mathbb Z$.\\

\noindent(4) $C=C_2\cup 2C_1$ (a union of a (smooth) conic and two lines). \\
We denote the intersection of the 2 lines $L_1,L_2$ by $P$.

(4.1) If the 2 lines and the conic are transversal and $P$ is not the intersection with the conic, the curve $C$ is nodal. By Proposition \ref{AbelianFund}, $\pi_1(\mathbb CP^2 \setminus (C_2\cup L_1\cup L_2))=\mathbb Z^2\oplus \mathbb Z/\text{gcd}(\text{deg}(C_2),\text{deg}(L_1),\text{deg}(L_2))\mathbb Z=\mathbb Z^2$.

(4.2) If the 2 lines and the conic are transversal and $P$ is the intersection with the conic, we blow up at $P$ once, to make the curves have normal crossings. Denote the simply-connected surface after blow-up of $\mathbb CP^2$ to be $X$, define $D$ to be the strict transform of the conic, and define $E$ to be the union of the (strict transforms of the) 2 lines and the exceptional divisor. Observe that $D^2=3>0=2r(D)$, where $r(D)$ is the number of nodes on $D$ (note here $D$ is irreducible). Then we can apply Nori's result (Proposition \ref{NormalCrossing}), to get the $\text{Ker }(\pi_1(X \setminus (D\cup E))=\pi_1(\mathbb CP^2 \setminus (C_2\cup L_1\cup L_2)) \to \pi_1(X \setminus E)=\pi_1(\mathbb CP^2 \setminus (L_1\cup L_2))=\mathbb Z)$ is abelian. It follows from Corollary \ref{cornew} that $\pi_1(\mathbb CP^2 \setminus (C_2\cup L_1\cup L_2))$ is virtually abelian.


(4.3) If one line is tangent to the conic, the other transverse, and $P$ is not on the conic, we apply the same method as the last case to blow up at the tangent point twice to make the curves have normal crossings. With the same notations, we have $D^2=2>0=2r(D)$, implying that $\text{Ker }(\pi_1(\mathbb CP^2 \setminus (C_2\cup L_1\cup L_2)) \to \pi_1(\mathbb CP^2 \setminus (L_1\cup L_2))=\mathbb Z)$ is abelian.
It follows from Corollary \ref{cornew} again that $\pi_1(\mathbb CP^2 \setminus (C_2\cup L_1\cup L_2))$ is virtually abelian.


(4.4) If one line is tangent to the conic, the other transverse, and $P$ is on the conic, as in the last case, we just need to blow up at $P$ twice to get the normal crossings condition. For the same reason, the fundamental group $\pi_1(\mathbb CP^2 \setminus (C_2\cup L_1\cup L_2))$ is virtually abelian.

(4.5) If the 2 lines are tangent to the conic and $P$ is not on the conic: In $\mathbb CP^2$, consider all lines through the point $P$. By the fact that the lines through $P$ are parametrized by $\mathbb CP^1$ {(Reason: Choose coordinates so that $P=[0:0:1]$ (since $\text{PGL}(3,\mathbb C)$ acts transitively on $\mathbb CP^2$), any line $ax+by+cz=0$ through $[0:0:1]$ would be actually given by $ax+by=0$. It is easy to see that these lines are parametrized by $[a:b]$)}, so we get a fibration $$\mathbb CP^1\setminus\{P\}\to \mathbb CP^2\setminus \{P\}\to \mathbb CP^1,$$ sending the line each line $V(ax+by)$ to $[a:b]$. Observe that most of the lines in $\mathbb CP^2$ through $P$ intersect $C$ with $3$ points (one at $p$ and two with the conic $K$), except the two tangent lines $C_1,C_1'$ of the conic passing through $P$. So if we remove $C$ from $\mathbb CP^2$, it induces a fibration $$\mathbb CP^1\setminus\{x_1,x_2,P\}\to \mathbb CP^2\setminus C\to \mathbb CP^1\setminus\{y_1,y_2\},$$ where $x_1,x_2$ are corresponding to the intersection of each line with the conic and $y_1,y_2$ are corresponding to the $C_1,C_1'$. Therefore, the fundamental group of $\mathbb CP^2\setminus C$ is the semi-direct product $F_2 \rtimes \mathbb{Z}$, which is known to be linear (actually every $F_2 \rtimes \mathbb{Z}$ group can be represented as a non-positively curved punctured torus bundle over the circle, implying that every $F_2 \rtimes \mathbb{Z}$ group is CAT(0) and virtually
special by the work of Liu \cite{liu}[Theorem 1.1]).


We summarize the results discussed above as the following. 

\begin{theorem}\label{degree4}
  For the reducible quartic case, the only non-abelian cases are $F_3, B_3,\mathbb Z\ast\mathbb Z/2\mathbb Z, F_2\rtimes \mathbb Z$, and virtually abelian.  
\end{theorem}Since each of the groups in Theorem \ref{degree4} is linear and virtually polyfree,  we conclude that:

\begin{cor}
    If $C$ is a plane curve of degree $\le 4$, then $\pi_1(\mathbb CP^2\setminus C)$ is linear and virtually polyfree.
\end{cor}

\section{Quintic}\label{degree5}
The presentation of fundamental groups of $\mathbb{C}P^2 \setminus\ C$ for degree-5 algebraic curves $C$ has been classified by Degtyarev \cite{Deg}. Based on this classification, we will prove Theorem \ref{main} for these curves by studying the presentations. 
\subsection{Presentation of fundamental groups}
In this subsection, we will review the results proved by Degtyarev \cite{Deg}. Before introducing the list of  Degtyarev, we need some notations.
\begin{enumerate}
\item $F_n$ is the free group of rank $n$, and $B_n$ is the braid group with $n$-strands.
\item $T_{p,q}$ is the fundamental group of a toric link of type $(p, q)$. In particular, if $p = 2$, then
$T_{2,2r}=\langle a,b| (ab)^r=(ba)^r\rangle$; if $\text{gcd}(p,q)=1$ then $T_{p,q}=\langle a,b| a^p=b^q\rangle$.
\item $G(T)$ and $G_p(T)$, where $T\in\mathbb Z[t]$ is an integral polynomial, are the extensions
$$1\to \mathbb Z[t]/T\to G(T)\to \mathbb Z\to 1$$ and
$$1\to (\mathbb Z/{p\mathbb Z})[t]/T\to G(T)\to \mathbb Z\to 1,$$ where the conjugation action of the generator of the quotient $\mathbb Z$ on the kernel is the multiplication by $t$.
\item $\text{Gr}\langle p,q,r \rangle :=\langle a,b,c|a^p=b^q=c^r=abc\rangle$.
\end{enumerate}

Following from Degtyarev's convention, we use Arnol'd's notation for the types of singular points (see \cite{arnold2012singularities}):

\begin{enumerate}
\item A curve is said to be type $aC_p\cup bC_q ...$ if it has $a$ irreducible components of degree $p$, $b$ irreducible of components of degree $q$ and so on. 
\item Let $A_p$ denote a singularity given locally by $x^2+y^{p+1}=0$ and a set of singularities is denoted like $3A_4 \cup 4A_3 \cup ....$ 
\item A curve $C$ is said of type $n_1C_p\cup n_2C_q\cup...$ if it has $n_1$ irreducible components of degree $p$,  $n_2$ irreducible components of degree $q$, etc. 
\item Let $C_d(\Sigma)$, where $\Sigma$ is a list of singularities, denote an irreducible curve of degree $d$ whose set of singular points is $\Sigma$.
\item The mutual position of an irreducible curve $C$ and a line $L$ is denoted by $\times d$, if $L$ meets $C$ with multiplicity $d$ at a non-singular point of $C$. Denoted by $A_p$ if $L$ intersects $C$ transversely at a singular point of $C$ of type $A_p$. Denoted by $A_p^*$ if $L$ is tangent to $C$ at a singular point of $C$ of type $A_p$.
\end{enumerate}

\subsection{Linearity and polyfreeness}
In this subsection, we will prove that the fundamental group of a quintic curve is virtually polyfree and linear.  Our key observation is that most of the non-abelian presentations are either surface groups or triangle Artin groups.

Here we list the quintics with nonabelian fundamental groups from \cite[Section 3.3]{Deg}:

\begin{enumerate}

\item
 Irreducible quintics of
type $C_5(3A_4 )$:  $\Pi=\langle a,b|b=ab^4a,a^2=b^2a^3b^2\rangle$ which is finite with order $320$, of type $C_{5}(A_6\sqcup 3A_2)$: $\Pi =\langle u,v|u^3=v^7=(uv^2)^2\rangle$.
\item The quintics of type $C_4\sqcup C_1$:\\
$\Pi =\langle a,b,c|aba=bab,bcb=cbc,abcb^{-1}a=bcb^{-1}abcb^{-1}\rangle$ (type $C_{4}(3A_{2})\sqcup \{\times 2,\times 2\}$), $B_m$ ($m=3,4)$, {$G_n(t+1)$} ($n=3,5$), $\text{Gr} \langle 2,3,5 \rangle \times \mathbb Z$, $T_{3,4}$.
\item The quintics of type $C_3\sqcup C_2$:\\
$\Pi =\langle a,b|[a^{3},b]=1,ab^{2}=ba^{2}\rangle$.
\item The quintics of type $C_3\sqcup 2C_1$:\\
$\mathbb Z\times B_3$; {$G(t^n-1),(n=2,3)$}; $T_{2,4},T_{2,6}$; {$\Pi=\langle a,b,c|aca=cac,[b,c]=1,(ab)^2=(ba)^2\rangle$} (type $C_3(A_2)\sqcup\{\times 3\}\sqcup\{\times 2,\times 1\}$).
\item The quintics of type $2C_2\sqcup C_1$:\\
$F_2; T_{2,4};\mathbb Z\times B_3$.
\item The quintics of type $C_2\sqcup 3C_1$:\\
$\mathbb Z\times F_2;\mathbb Z\times T_{2,4}$;
{$\Pi=\langle a,b,c|[a,b]=[a,c^{-1}bc]=1,(bc)^2=(cb)^2\rangle$;
$\Pi=\langle a,b,c|(ac)^2=(ca)^2,(ab)^2=(ba)^2,[b,c]=1\rangle$.}
\item The quintics of type $5C_1$.\\
$F_4;\mathbb Z\times F_3;F_2\times F_2;\mathbb Z\times \mathbb Z\times F_2$.
\end{enumerate}

It's well-known that finite groups, free groups, braid groups, the direct product of two linear (resp. polyfree) groups, and virtually linear (resp. polyfree) groups are linear (resp. virtually polyfree). Therefore, it remains to check the linearity and virtual polyfreeness of the following groups.


\begin{enumerate}
\item Type $C_{5}(A_6\sqcup 3A_2)$, whose proof follows from Lemma \ref{C5(A63A2)1} and Lemma \ref{C5(A63A2)2}.
\item Type $C_{4}(3A_{2})\sqcup \{\times 2,\times 2\}$), whose proof follows from Lemma \ref{C4(3A2)22)}.
\item $G_n(t+1)$ ($n=3,5$). This is a finite-by-$\mathbb Z$ group, which is virtually $\mathbb Z$ and thus linear, and virtually polyfree.
\item $\text{Gr} \langle 2,3,5 \rangle$, whose proof follows from Lemma \ref{Gr235}. 
\item $T_{3,4}$, $T_{2,4}$, and $T_{2,6}$, whose proof follows from Lemma \ref{toriclink}.
\item Type $C_3\sqcup C_2$, whose proof follows from Lemma \ref{C3C2}.
\item Type $C_3\sqcup 2C_1$. Observe that $\Pi=\langle a,b,c|aca=cac,[b,c]=1,(ab)^2=(ba)^2\rangle$ is indeed the triangle Artin group $\mathrm{Art}_{234}$,  which is also the Artin group of (finite) type $B_3$. 
It is known that any finite-type Artin group is linear (see \cite{CW}). The group $\mathrm{Art}_{234}$ is isomorphic to the semi-direct product $\mathrm{Art}_{333} \rtimes \mathbb{Z}$ (cf. Charney-Crisp \cite[page 6]{CC}).  Squier \cite{Sq} proved that the affine triangle Artin group $\mathrm{Art}_{333}$ is polyfree, implying that $\mathrm{Art}_{234}$ is polyfree.

\item $G(t^n-1),(n=2,3)$. This is a $\mathbb Z^n$-by-$\mathbb Z$ group. Note that $G(t^n-1)$ is linear since it is polycyclic (cf. Segal \cite[Theorem 5, page 92]{Segal}).
\item $\Pi=\langle a,b,c|[a,b]=[a,c^{-1}bc]=1,(bc)^2=(cb)^2\rangle$, whose proof follows from Lemma \ref{C23C1}.
\item $\Pi=\langle a,b,c|(ac)^2=(ca)^2,(ab)^2=(ba)^2,[b,c]=1\rangle$. This is the Artin group $\mathrm{Art}_{244}$, which is of affine type $\tilde C_2$. It is known to be linear and polyfree as explained in Section 2.2. 
\end{enumerate}

\begin{lemma}\label{C5(A63A2)1}
\lbrack type $C_{5}(A_6\cup 3A_2)$] The group 
\begin{equation*}
\Pi =\langle u,v|u^3=v^7=(uv^2)^2\rangle
\end{equation*}%
is virtually a $\mathbb Z$-central extension of the surface group $\pi_1(S_3)$ for the orientable closed surface $S_3$ of genus $3$.
\end{lemma}
\begin{proof}
Observe that $\Pi$ is a $\mathbb Z$-central extension of $$\langle u,v|u^3=v^7=(uv^2)^2=1\rangle,$$ where the center is generated by $\langle u^3\rangle\subset \Pi$.
Note that $\langle u,v|u^3=v^7=(uv^2)^2=1\rangle$ is indeed the triangle group $\Delta_{(2,3,7)}:=\langle a,b|a^2=b^3=(ab)^7 =1\rangle$ (by the isomorphism $\phi:\Delta_{(2,3,7)}\to\Pi, a\mapsto uv^2,b\mapsto u$). By \cite{behr1968presentation}, $\Delta_{(2,3,7)}$ admits a surjection onto the finite group $\text{PSL}(2,\mathbb Z/7\mathbb Z)$ with kernel $\pi_1(S_3)$. So $\Delta_{(2,3,7)}$ is virtually a surface group. It follows that $\Pi$ is virtually a $\mathbb Z$-central extension of $\pi_1(S_3)$.
\end{proof}

\begin{lemma} \label{C5(A63A2)2}
Let $S_g$ be the closed orientable surface of genus $g$.    Any $\mathbb Z$-central extension of the surface group $\pi_1(S_g)$ is linear.
\end{lemma}
\begin{proof}
   Let $$1\to \mathbb Z=\langle t\rangle\to G\to \langle a_1,...,a_g,b_1,...,b_g|\prod_{i=1}^n[a_i,b_i]=1\rangle \to 1$$ be a $\mathbb Z$-central extension of $\pi_1(S_g)$. 
   Then $G$ has a presentation 
   $$\langle A_1,...,A_g,B_1,...,B_g,t|\prod_{i=1}^n[A_i,B_i]=t^p\rangle,$$ for some $p\in\mathbb Z$, where $A_i$ (resp. $B_i$) is representative of $a_i$ (resp. $b_i$) in $G$.
   Observe that the quotient 
   \begin{eqnarray*}    
   & G/\langle\langle A_2,...,A_g,B_2,...,B_g\rangle\rangle &\\
  & \cong \langle A_1,B_1,t|[A_1,B_1]=t^p,[A_1,t]=[B_1,t]=1,[A_i,t]=1,[B_i,t]=1 \rangle
   \end{eqnarray*}
    is nilpotent and thus linear.
   We consider two quotient maps $$f_1:G\to \pi_1(S_g), f_2:G\to G/\langle\langle A_2,...,A_g,B_2,...,B_g\rangle\rangle.$$ Note that $\ker f_1$ is the center $\mathbb Z$, which is non-trivially mapped by $f_2$. Therefore, the product map $$f_1 \times f_2: G \rightarrow \pi_1(S_g) \times G/\langle\langle A_2,...,A_g,B_2,...,B_g\rangle\rangle $$ is injective into a product of linear groups, and the result follows.
\end{proof}
Note that a surface group is polyfree, implying that a $\mathbb{Z}$-extension of a surface group is polyfree.

\begin{lemma}\label{C4(3A2)22)}
\lbrack type $C_{4}(3A_{2})\sqcup \{\times 2,\times 2\}$] The group 
\begin{equation*}
\Pi =\langle a,b,c|aba=bab,bcb=cbc,abcb^{-1}a=bcb^{-1}abcb^{-1}\rangle
\end{equation*}%
is isomorphic to the Artin group $\mathrm{Art}_{333}$, thus linear and polyfree.
\end{lemma}

\begin{proof}
Let $x=bcb^{-1}.$ Rewrite the presentation of $\Pi $ with generators $a,b,x.$
Then $c=b^{-1}xb.$ The relator $abcb^{-1}a=bcb^{-1}abcb^{-1}$ becomes $%
axa=xax.$ The relator $bcb=cbc$ becomes 
\begin{eqnarray*}
bb^{-1}xbb &=&b^{-1}xbbb^{-1}xb, \\
xbb &=&b^{-1}xbxb, \\
bxb &=&xbx.
\end{eqnarray*}%
Therefore, the group $\Pi $ is isomorphic to 
\begin{equation*}
\langle a,b,x|aba=bab,bxb=xbx,axa=xax\rangle ,
\end{equation*}%
which is the Artin group $Art_{333}.$
\end{proof}

\begin{lemma}\label{Gr235}
    $\text{Gr}\langle 2,3,5 \rangle =\langle a,b,c|a^2=b^3=c^5=abc\rangle$ is virtually free and linear.
\end{lemma}
\begin{proof}
    Observe that $\langle a^2\rangle\subset \text{Gr}\langle 2,3,5 \rangle$ is a central cyclic subgroup. Thus, we have the short exact sequence:
    $$1\to \mathbb Z=\langle a^2\rangle\to \text{Gr}\langle 2,3,5 \rangle\to G:=\langle a,b,c|a^2=b^3=c^5=abc=1\rangle\to 1.$$
    Observe that $G=\langle b,c|b^3=c^5=(bc)^2=1\rangle\cong\langle x,y|x^2=y^3=(xy)^5=1\rangle$, where $x=bc, y=b$. Therefore, the group $G$ is actually the alternating group $A_5$, implying that $\text{Gr}\langle 2,3,5 \rangle$ is virtually free and linear.
\end{proof}

\begin{lemma}\label{toriclink}
    The toric link groups $T_{p,q}=\langle a,b|a^p=b^q\rangle$, where $\text{gcd}(p,q)=1$ and 
    $T_{2,2r}=\langle a,b|(ab)^r=(ba)^r\rangle$
    are linear and virtually polyfree.
\end{lemma}
\begin{proof}
    For $T_{p,q}$, observe that $\langle a^p\rangle\subset T_{p,q}$ is a central cyclic subgroup. Therefore, we have the short exact sequence
    $$1\to \mathbb Z=\langle a^p\rangle\to T_{p,q}\to G:=\langle a,b|a^p=b^q=1\rangle\to 1.$$
    Note that the quotient group $G\cong \mathbb Z/p\mathbb Z\ast \mathbb Z/q\mathbb Z$, is virtually free. This implies that $T_{p,q}$ is virtually a central cyclic extension of a free group $F$. Note that for a central cyclic extension of $F$, the extensions are classified by the second cohomology group $H^2(F,\mathbb Z)$, which is trivial since $F$ is free. Therefore, the group $T_{p,q}$ is virtually a direct product of $\mathbb Z$ and $F$, hence linear.

    For $T_{2,2r}$, similarly we have the short exact sequence
    $$1\to \langle (ab)^r\rangle\to T_{2,2r}\to H:=\langle a,b|(ab)^r=(ba)^r=1\rangle\to 1.$$ Observe that $(ab)^r=1$ implies $(ba)^r=a^{-1}a(ba)^r=a^{-1}(ab)^ra=1$. Therefore, we have an isomorphism $H\cong\langle a,c|c^r=1\rangle\cong \mathbb Z\ast\mathbb Z/r\mathbb Z$, where $c:=ab$. The group $H$ is virtually free. With a similar argument as $T_{p,q}$, the group $T_{2,2r}$ is linear.
    The same exact sequences show that the groups are virtually polyfree.
\end{proof}


\begin{lemma}\label{C3C2}
\lbrack type $C_{3}\sqcup C_{2}$] The group%
\begin{equation*}
\Pi =\langle a,b|[a^{3},b]=1,ab^{2}=ba^{2}\rangle
\end{equation*}%
is virtually $\mathbb{Z}^{2}.$
\end{lemma}

\begin{proof}
It is clear that $a^{3}$ lies in the center of $\Pi .$ Since $a^{3}b=ba^{3},$
we have%
\begin{eqnarray*}
a^{3}b &=&ba^{2}a=ab^{2}a, \\
a^{2}b &=&b^{2}a, \\
a^{2}ba^{-1} &=&b^{2}, \\
(ab)a^{2}ba^{-1} &=&(ab)b^{2}, \\
a(ba^{2})ba^{-1} &=&ab^{3}, \\
a(ab^{2})ba^{-1} &=&ab^{3}a^{-1}(a), \\
a(ab^{3}a^{-1}) &=&(ab^{3}a^{-1})a, \\
ab^{3}a^{-1} &=&b^{3},
\end{eqnarray*}%
which proves that $b^{3}$ also lies in the center of $\Pi .$ In the quotient
group $\Pi /\langle a^{3},b^{3}\rangle ,$ the relator $ab^{2}=ba^{2}$
implies that $ab^{-1}=ba^{-1}$ and $(ab^{-1})^{2}=1.$ This shows that the
natural homomorphism%
\begin{eqnarray*}
\langle x,y|x^{3} &=&1,y^{3}=1,(xy)^{2}=1\rangle \rightarrow \Pi /\langle
a^{3},b^{3}\rangle \\
x &\mapsto &a, \\
y &\mapsto &b^{-1}
\end{eqnarray*}%
is surjective. Since $\langle x,y|x^{3}=1,y^{3}=1,(xy)^{2}=1\rangle $ is an index-2 subgroup of the
finite Coxeter group $\mathrm{Cox}_{233}$, we know that $\Pi /\langle a^{3},b^{3}\rangle $ is
finite.
\end{proof}


\begin{lemma}\label{C23C1}
\lbrack type $C_{2}\sqcup 3C_{1}$] The group%
\begin{equation*}
\Pi =\langle a,b,c|[a,b]=[a,c^{-1}bc]=1,(bc)^{2}=(cb)^{2}\rangle
\end{equation*}%
has an index-2 right-angled Artin subgroup.
\end{lemma}

\begin{proof}
Let $x=bc.$ Then $c=b^{-1}x.$ The group $\Pi $ has a presentation with
generators $a,b,x.$ The relator $[a,b]=1$ is equivalent to%
\begin{equation*}
bab^{-1}=a.
\end{equation*}%
The relator $[a,c^{-1}bc]=1$ becomes%
\begin{eqnarray*}
ac^{-1}bc &=&c^{-1}bca, \\
ax^{-1}bx &=&x^{-1}bxa, \\
xax^{-1} &=&b(xax^{-1})b^{-1}.
\end{eqnarray*}%
The relator $(bc)^{2}=(cb)^{2}$ becomes%
\begin{eqnarray*}
x^{2} &=&b^{-1}(x^{2})b, \\
bx^{2}b^{-1} &=&x^{2}.
\end{eqnarray*}%
Therefore, $\Pi $ has a new presentation%
\begin{equation*}
\langle
a,x,b|bab^{-1}=a,b(xax^{-1})b^{-1}=xax^{-1},bx^{2}b^{-1}=x^{2}\rangle ,
\end{equation*}%
which is an HNN extension of free group $\langle a,x\rangle $ by the stable
letter $b$ along the subgoup $\langle a,xax^{-1},x^{2}\rangle .$ Define 
\begin{eqnarray*}
\phi &:&\langle
a,x,b|bab^{-1}=a,b(xax^{-1})b^{-1}=xax^{-1},bx^{2}b^{-1}=x^{2}\rangle
\rightarrow \mathbb{Z}/2 \\
&a,b &\mapsto 0, \\
&x &\mapsto 1.
\end{eqnarray*}%
Choose $R=\{1,x\}$ as the representative of left cosets for $\ker \phi <\Pi
. $ The Reidemeister-Schreier method gives a presentation of $\ker \phi $
with generators 
\begin{equation*}
a,b,a^{\prime }:=xax^{-1},b^{\prime }:=xbx^{-1},t=x^{2}
\end{equation*}%
and relators 
\begin{eqnarray*}
bab^{-1} &=&a,ba^{\prime }b^{-1}=a^{\prime },btb^{-1}=t, \\
b^{\prime }a^{\prime }b^{\prime -1} &=&a^{\prime },b^{\prime }tb^{\prime
-1}=t^{},b^{\prime }ab^{\prime -1}=a,
\end{eqnarray*}%
where the last relator comes from the $x$-conjuate to $%
b(xax^{-1})b^{-1}=xax^{-1}$ (i.e. $b^{\prime }(xxax^{-1}x^{-1})b^{\prime
-1}=x^{2}ax^{-2},$ which is equivalent to $b^{\prime }ab^{\prime -1}=a$
using $b^{\prime }tb^{\prime -1}=t^{}$). It is direct that the
presentation shows that $\ker \phi$ is a right-angled Artin group.
\end{proof}



\bibliographystyle{alpha}
\bibliography{sample}

\end{document}